\documentclass[12pt,a4paper]{amsart}%

\usepackage{amsmath}
\usepackage{fullpage}
\usepackage{MyCommA}
\usepackage[final]{pdfpages}

\newcommand{\NextVer}[1]{}

\begin{document}
	
	\author[Aizenbud]{Avraham Aizenbud}
	\address{Avraham Aizenbud,
		Faculty of Mathematical Sciences,
		Weizmann Institute of Science,
		76100
		Rehovot, Israel}
	\email{aizenr@gmail.com}
	\urladdr{https://www.wisdom.weizmann.ac.il/~aizenr/}

\author{Uri Bader}
\address{Uri Bader, Faculty of Mathematics and Computer Science, Weizmann Institute of Science, Israel.}
\email{uri.bader@gmail.com}
\urladdr{https://www.weizmann.ac.il/math/uribader/home}

	\date{\today}
	
	%
	%
	%
	%
	%
	%
	%
	%
	

\title{Geometric representations of group actions}
\maketitle
\begin{abstract} 
We study equivariant morphisms from zero dimensional schemes to varieties
and show that, under suitable assumptions, 
all such morphisms factor via a canonical one.
We relate the above to Algebraic Representations of Ergodic Actions.
\end{abstract}


\section{Introduction}

Fix a field $k$ and a $k$-algebraic group $\bfG$.
In this paper, by a variety we mean a geometrically reduced separated $k$-scheme of finite type (not necessarily irreducible). We consider the category $\mathcal{C}$ of varieties endowed with a $k$-morphic action by $\bfG$, 
and their $\bfG$-equivariant $k$-morphisms.
Assume given an object $\bfX_0$ in $\mathcal{C}$ which is $\bfG$-transitive and
consider the algebra of $k$-rational functions, $K=k(\bfX_0)$,
on which $G=\bfG(k)$ acts by $k$-automorphisms.
One might ask how to reconstruct $\bfX_0$, given the the $k$-algebra $K$ along with the $G$-action.
Noting that the natural morphism of $k$-schemes $\phi:\Spec(K)\to \bfX_0$ is $G$-equivariant
we suggest the following answer.

\begin{introprop} \label{prop:field}
Assume that $G$ is Zariski dense in $\bfG$.
Then the functor $F_K:\mathcal{C}\to \mbox{Sets}$
which assigns to $\bfX$
the set of $G$-equivariant morphisms of $k$-schemes $\Spec(K)\to \bfX$
is represented by $\bfX_0$ and $\phi$ corresponds to the identity morphism $\bfX_0\to \bfX_0$.
\end{introprop}

It is natural to widen up the scope as follows.
We fix a group $\Gamma$ and a group homomorphism $\Gamma\to G$.
Assume given a $k$-scheme $\bfY$ (not necessarily of finite type) endowed with an action of $\Gamma$ by $k$-automorphisms
and consider the functor $F_\bfY:\mathcal{C}\to \mbox{Sets}$
which assigns to $\bfX$ the set of $\Gamma$-equivariant morphisms of $k$-schemes $\bfY\to \bfX$.
If this functor is representable we will say that \emph{$\bfY$ is representable in $\cC$}.
In case $\bfY=\Spec(L)$, for some $k$-algebra $L$ endowed with an action of $\Gamma$,
we will write $F_L$ for $F_\bfY$ and say that \emph{$L$ is representable in $\cC$} if $\bfY$ is.
We are led to consider the following general problem.

\begin{introproblem} \label{ques:scheme}
Find natural criteria for representability in $\cC$ of schemes and algebras on which $\Gamma$ acts.
\end{introproblem}

We will see in the course of proof of Proposition~\ref{prop:field} that the algebra $K=k(\bfX_0)$ considered in Proposition~\ref{prop:field}
is zero dimensional and geometrically reduced.
Other prominent examples of zero dimensional and geometrically reduced algebras are algebras of $k$-valued measurable functions on measurable spaces, 
where $k$ is a local field.
The question of representability of such algebras of measurable functions 
is closely related to the study of \emph{algebraic representations of ergodic actions},
considered in \cite{BF20} and \cite{BaderFurman}.
In fact, this study was our initial motivation for considering Problem~\ref{ques:scheme}.
The following theorem, which is the main result of this paper,
completely solves this problem, under the zero dimensionality and geometrically reducedness assumptions.
Here we denote by $\cO(\bfY)$ the $k$-algebra of
global regular functions on $\bfY$ and by $L^\Gamma$ the subalgebra of $\Gamma$-invariants in $L$.

\begin{introthm} \label{thm:E}
Assume $\bfY$ is a zero dimensional and geometrically reduced $k$-scheme
endowed with an action of $\Gamma$.
Then the following are equivalent.
\begin{itemize}
\item $\bfY$ is representable in $\cC$.
\item $\cO(\bfY)$ is representable in $\cC$.
\item $\cO(\bfY)^\Gamma$ is finite over $k$.
\item $\cO(\bfY)^\Gamma$ is a finite sum of finite  separable field extensions of $k$. 
\end{itemize}
Moreover, if $\cO(\bfY)$ is representable by
$\Spec(\cO(\bfY)) \to \bfX$ then $\bfY$ is represented by the composition with the canonical map,
$\bfY\to \Spec(\cO(\bfY)) \to \bfX$.

In particular, for a zero dimensional and geometrically reduced $k$-algebra $L$,
$L$ is representable in $\cC$ iff $L^\Gamma$ is finite over $k$ iff $L^\Gamma$ is a finite sum of finite  separable field extensions of $k$.
\end{introthm}

In \S\ref{sec:zd} we will give some preliminaries regarding zero dimensional rings and algebras.
A fundamental concept to this paper is a \emph{geometrization} of a scheme
which will be introduced and discussed in \S\ref{sec:geo-cat} and will be extended to an equivariant setting in \S\ref{sec:geo-eq}.
The proofs of Theorem~\ref{thm:E} and Proposition~\ref{prop:field} will be given in \S\ref{sec:proofs}.
Finally, in \S\ref{sec:meas} we will relate all of the above to algebraic representations of ergodic actions, considered in \cite{BF20} and \cite{BaderFurman}.

%
%
%

\subsection{Acknowledgements} 
A.A. was partially supported by ISF grant number 249/17 and a Minerva
foundation grant. 
U.B was partially supported by ISF Moked grant number 2919/19.

\section{Zero dimensionality} \label{sec:zd}

In this paper, and in particular in this section, all rings are commutative and unital.

\subsection{Zero dimensional rings}

The following two lemmas, which are taken from \cite{BrRi}, are fundamental to our analysis of zero dimensional algebras.
For convenience, we include their short proofs.

\begin{lemma} \label{lem:usup}
Let $L$ be a ring. 
Given $x\in L$, there exists at most one idempotent $e\in L$ such that  
$ex$ is a unit in $eL$ and $(1-e)x$ is nilpotent in $L$.
\end{lemma}

Given $x\in L$, an idempotent $e$ satisfying the properties above is said to be \emph{the supporting idempotent of $x$}.

\begin{proof}
If $e$ and $e'$ are both supporting idempotents for $x$ then there exist $y\in L$ and $n\in \mathbb{N}$ such that $e=exy$ 
and $(1-e')x^n=0$.
It follows that
\[ e=e^n=(exy)^n= e'(exy)^n+(1-e')(exy)^n=e'(exy)^n=e'e, \]
thus $e=e'e$. Symmetrically, we get that $e'=ee'$, thus $e=e'$.
\end{proof}

If $x\in L$ has a supporting idempotent we will denote it 
by $e_x$.

\begin{lemma} \label{lem:zdidem}
A ring $L$ is a zero dimensional
iff every $x\in L$ has a supporting idempotent.
\end{lemma}

\begin{proof}
Assume $L$ has the property that every element has a supporting idempotent
and note that this property 
passes to the quotients of $L$, thus in order to prove that $L$ is zero dimensional it is enough to assume that $L$ is
a domain and show that it is a field. In a domain the only idempotents are 0 and 1.
If $e_x=0$ then $x$ is nilpotent, hence $x=0$
and if $e_x=1$ then $x$ is invertible. Thus, indeed, $L$ is a field.

Assume now that $L$ is zero dimensional and fix $x\in L$.
Then there exist $y\in L$ and $n\in \mathbb{N}$ such that $x^{n+1}y=x^n$.
Indeed, otherwise 0 is not included in the multiplicative set 
\[ S=\{(1-xy)x^n\mid y\in L,~n\in \mathbb{N} \} \]
which contains $x$,
thus we can find a prime ideal $\mathfrak{p}\lhd L$ disjoint of $S$,
which is maximal by zero dimensionality,
and, since $x\notin \mathfrak{p}$, we get a contradiction by picking $y=x^{-1}$ mod $\mathfrak{p}$.
Note that for every $k\in \mathbb{N}$, $x^{n+k}y^k=x^n$
thus, in particular, $x^{2n}y^n=x^n$.
We set $e=x^ny^n$ and observe that $e$ is idempotent, as 
\[ e^2=x^{2n}y^{2n}=(x^{2n}y^n)y^n=x^ny^n=e \]
and $e$ is supporting for $x$, as 
\[ exy=(xy)^nxy=(x^{n+1}y)y^n=x^ny^n=(xy)^n=e \]
and
\[ ((1-e)x)^n=(1-e)x^n=(1-x^ny^n)x^n=x^n-x^{2n}y^n=x^n-x^n=0. \]
\end{proof}

\begin{lemma} \label{lem:LGamma}
    If $L$ is a zero dimensional ring and $\Gamma$ is a group acting on $L$ by automorphism then the ring of invariant $L^\Gamma$ is zero dimensional.
\end{lemma}

\begin{proof}
    This follows from Lemma~\ref{lem:zdidem}, as for every $\Gamma$-invariant element, the supporting idempotent is obviously $\Gamma$-invariant as well.
\end{proof}

We note that the collection of idempotents in an algebra forms a lattice, where
$e_1\wedge e_2=e_1e_2$ and $e_1\vee e_2=e_1+e_2-e_1e_2$.

\begin{lemma} \label{lem:sg}
In a reduced zero dimensional ring,
every finitely generated ideal is generated by a single idempotent.
\end{lemma}

\begin{proof}
The ideal generated by $x_1,\ldots,x_n$ is generated by $e_{x_1} \vee \cdots \vee e_{x_n}$.
\end{proof}

Note that reduced zero dimensional rings are also known as \emph{von Neumann regular rings}. 
In the sequel we will be mostly interested in zero dimensional algebras over a field which are geometrically reduced.

\subsection{Zero dimensional algebras}

We now fix a field $k$.
We observe that zero dimensionality is not preserved by base change.

\begin{example}
The $k(x)$ algebra $k(x)\otimes_k k(y)$ is not zero dimensional.
Indeed, identifying it with its image in $k(x,y)$ we see that it is an integral domain, but $x+y$ is not invertible.  
\end{example}

However, zero dimensionality is preserved by a finite base change.

\begin{lemma} \label{lem:zdext}
    If $L$ is zero dimensional $k$-algebra and $k<k'$ is a finite field extension then $L' =k' \otimes_k L$ is also a zero dimensional $k$-algebra.
\end{lemma}

\begin{proof}
    We fix a prime ideal $p'\lhd L'$ and argue to show that $L'/p'$ is a field.
    We denote $p=p'\cap L$ and note that $L/p$ is a field and
    $k' \otimes_k L/p$ is a finite dimensional vector space over $L/p$.
    Hence its quotient $L'/p'$ is also finite dimensional vector space over $L/p$.
    Since $L'/p'$ is an integral domain, we conclude that it is indeed a field. 
\end{proof}

\begin{lemma} \label{lem:finitealg}
    A geometrically reduced zero dimensional $k$-algebra $L$ is finite over $k$ iff it is a finite sum of finite separable field extensions of $k$.
\end{lemma}

\begin{proof}
    The ``if" direction is clear. For the ``only if" direction, observe that if $L$ is over $k$ than it has finitely many idempotents, thus we assume as we may that it has a single idempotent, in which case it is a filed, which must be separable over $k$ by the assumption that $L$ is geometrically reduced.
\end{proof}

\section{Geometrizations} \label{sec:geo-cat}

We fix a field $k$
and consider the category of $k$-schemes and their morphisms.
$k$-schemes, that is schemes over $k$, will be denoted by boldface letters.
By definition they come equipped with a scheme morphism to $\Spec(k)$ which our notation will suppress.
Typically, we will also suppress the reference to $k$ when discussing $k$-morphisms.
Similarly, all algebras, vector spaces etc., as well as their morphisms,
will be tacitly assumed to be over $k$, unless stated otherwise.
We fix throughout this section a scheme $\bfY$ (not necessarily of finite type)
and we will consider morphisms from $\bfY$ to various varieties.
To be clear, we use the following convention.

\begin{defn} \label{def:variety}
By a variety we mean a geometrically reduced separated $k$-scheme of finite type (not necessarily irreducible).
\end{defn}

A morphism $\bfY\to \bfX$ such that $\bfX$ is a variety is called a \emph{geometrization} of $\bfY$.
A \emph{morphism of geometrizations} of $\bfY$ from $\phi_1:\bfY\to \bfX_1$ to $\phi_2:\bfY\to\bfX_2$
is a morphism $\theta:\bfX_1\to\bfX_2$ such that $\phi_2=\theta\circ \phi_1$.
Thus, the category of geometrizations of $\bfY$ is the full subcategory of the corresponding coslice category (a.k.a. under category),
where the codomains are varieties.
The category of geometrizations of $\bfY$ has a final object, the
\emph{trivial geometrization} $\bfY\to \spec(k)$.  
We note that products of varieties are varieties, by \cite[Lemma 035Z]{SP},
and observe that the category of geometrizations of $\bfY$ has products, the product of 
$\phi_1:\bfY\to \bfX_1$ and $\phi_2:\bfY\to\bfX_2$ being $\phi_1\times \phi_2:\bfY\to \bfX_1\times \bfX_2$.

We now consider the full subcategory of dominant geometrizations of $\bfY$, that is dominant morphisms $\bfY\to \bfX$ where $\bfX$ is a variety. 

\begin{prop} \label{prop:D}
    If $\bfY$ is a geometrically reduced scheme then the inclusion of the category of dominant geometrizations of $\bfY$
    in the category of all geometrizations of $\bfY$ has a right adjoint functor $D$.
\end{prop}

\begin{proof}
    Given a geometrization $\phi:\bfY\to \bfX$, we set $\bfX'$ to be scheme theoretic image of $\phi$, which is the smallest closed subscheme of $\bfX$ through which $\phi$ factors, see \cite[Definition 01R7]{SP}, and we set $D(\phi):\bfY \to \bfX'$ to be the corresponding morphism.
    Since $\bfY$ is reduced, $D(\phi)$ is dominant and $\bfX'$ is reduced by \cite[Lemma 056B]{SP}.
    By the proof of \cite[Lemma 01R6]{SP}, the morphisms of the local rings of $\bfX'$ into the local rings of $\bfY$ are injective. Using \cite[Lemma 035W]{SP} and \cite[Lemma 030T]{SP}, we get that $\bfX'$ is in fact geometrically reduced.
 The functoriality of $D$ and the fact that it is right adjoint to the inclusion functor is straightforward.
\end{proof}

 \begin{cor} \label{cor:prod}
    If $\bfY$ is a geometrically reduced scheme then the category of dominant geometrizations of $\bfY$ has a product,
    given by $\phi_1\times_D \phi_2:= D(\phi_1\times \phi_2)$
 \end{cor}

\begin{remark} \label{rem:prod}
Given a category, a right adjoint functor to the inclusion of a full subcategory is called a \emph{contraction}.
In the presence of a contraction, if the given (ambient) category has products, then a formula as in Corollary~\ref{cor:prod} defines products in the subcategory, 
as one can easily check.
\end{remark}

\begin{prop} \label{prop:clopen}
Let $\bfY$ be a zero dimensional scheme
and let $\phi:\bfY\to \bfX$ be a geometrization.
Then for every constructible subvariety $\bfZ\subset \bfX$, $\phi^{-1}(\bfZ)$ is clopen in $\bfY$,
corresponding to an idempotent $e_{\bfZ}\in \cO(\bfY)$.
\end{prop}

\begin{proof}
Since being clopen is a local condition, we assume as we may that $\bfY$ is affine,
corresponding to a reduced zero dimensional algebra $L$.
Since the collection of subsets in $\bfX$ which preimage is clopen forms a Boolean algebra,
we assume as we may that $\bfZ$ is a closed subvariety.
We note that the inclusion morphism $\bfZ\subseteq \bfX$ is a closed immersion of finite presentation 
(see \cite[Definition 01TP]{SP}).
We form the base change $\bfY\times_\bfX \bfZ\to \bfY$ and note that it is 
a closed immersion of finite presentation by \cite[Lemma 01QR and Lemma 01TS]{SP}.
We conclude by \cite[Lemma 01TV]{SP} that ideal in $L$ associated with $\bfY\times_\bfX \bfZ$ is finitely generated.
We are done by Lemma~\ref{lem:sg}.
\end{proof}

\begin{remark} \label{rem:projmeas}
In case $\bfY=\Spec(L)$, we conclude that for each constructible subvariety $\bfZ \subseteq \bfX$,
$\phi^{-1}(\bfZ)$ is the zero set of the idempotent $e_\bfZ \in L$.
The correspondence $\bfZ\mapsto e_\bfZ$ could be thought of as a projection valued measure on $\bfX$.
\end{remark}

\begin{prop}
If $\bfY$ is a zero dimensional scheme then $\cO(\bfY)$ is a zero dimensional algebra.
\end{prop}

\begin{proof}
    For $f\in \cO(\bfY)$ we consider the associated geometrization $\bfY \to \mathbb{A}^1$, the constructible (open) subvariety $\bfZ =\mathbb{A}^1-\{0\}\subset \mathbb{A}^1$ and the corresponding idempotent $e_\bfZ \in \cO(\bfY)$.
    Since $f$ is invertible on the clopen $f^{-1}(\bfZ)\subseteq \bfY$, we are done by Lemma~\ref{lem:usup}.
\end{proof}

\begin{prop} \label{pro:zdaffine}
If $\bfY$ is a zero dimensional scheme then every geometrization of $\bfY$ factors uniquely via the canonical map $\bfY\to \Spec(\cO(\bfY))$.
\end{prop}

\begin{proof}
The statement is obvious if $\bfX$ is affine.
Assume $\bfU_i$ is a finite open affine cover of $\bfX$,
define $\bfZ_i=\bfU_i-\cup_{j< i} \bfU_j$
and set $\bfY_i=\phi^{-1}(\bfZ_i)$.
For every $i$, $\bfZ_i\subseteq \bfX$ is an affine subvariety
thus $\phi:\bfY_i\to \bfZ_i$ factors uniquely via the canonical map $\bfY_i \to \Spec(\cO(\bfY_i))$.
Since $\bfZ_i$ form a constructible partition of $\bfX$,
$\bfY_i$ form a clopen partition of $\bfY$, thus $\cO(\bfY)=\oplus_i\cO(\bfY_i)$
and $\Spec(\cO(\bfY)$ is the disjoint union of its clopen subsets $\Spec(\cO(\bfY_i))$.
The result follows.
\end{proof}

\section{Equivariant geometrizations} \label{sec:geo-eq}

In this section we fix a field $k$ and a $k$-algebraic group $\bfG$, that is a group scheme which is a $k$-variety (see Definition~\ref{def:variety}).
We consider the category $\mathcal{C}$ of $k$-varieties 
endowed with a $k$-morphic action by $\bfG$, 
and their $\bfG$-equivariant $k$-morphisms.
We fix a group $\Gamma$ and a group homomorphism $\Gamma\to \bfG(k)$.
We fix a $k$-scheme $\bfY$ (not necessarily of finite type) endowed with an action of $\Gamma$ by $k$-automorphisms.
We also fix once and for all an algebraically closed field extension $k<\bar{k}$ and we assume it is large enough so that $\bfY$ admits a geometric $\bar{k}$-point, $\Spec(\bar{k})\to \bfY$.

We denote by $\Geom(\bfY)$ the category of \emph{equivariant} geometrizations of $\bfY$,
that is $\Gamma$-equivariant morphisms of $k$-schemes $\phi:\bfY\to \bfX$, where $\bfX$ is an object in $\cC$.
The equivariant geometrization $\phi$ is said to be \emph{essential} if the associated geometrization $\phi':\bfG \times \bfY \to \bf X$, given by the composition
\[ \bfG \times \bfY \to \bfG \times \bfX \to \bfX, \]
where the map $\bfG \times \bfX \to \bfX$ is the action map,
is dominant.
The full subcategory of essential geometrizations is denoted $\Geom^E(\bfY)$.

\begin{prop} \label{prop:E}
    If $\bfY$ is a geometrically reduced scheme then the inclusion of categories $\Geom^E(\bfY) \subseteq \Geom(\bfY)$
has a right adjoint functor $E$.
\end{prop}

\begin{proof} 
 Given $\phi:\bfY\to \bfX$ in $\Geom(\bfY)$ we consider the associated geometrization of $\phi':\bfG \times \bfY \to \bf X$ introduced above and the geometrization $D(\phi'):\bfG \times \bfY \to \bfX'$, given in Proposition~\ref{prop:D}. Considering the obvious morphism $i:\bfY\to \{e\}\times \bfY \subseteq \bfG\times \bfY$, we set $E(\phi)=D(\phi') \circ i$.
 Given this construction, one easily checks the details of the proof.  
\end{proof}

As in the introduction, we consider the functor $F_\bfY:\mathcal{C}\to \mbox{Sets}$
which assigns to $\bfX$ the set of equivariant geometrizations $\bfY\to \bfX$.
If this functor is representable we will say that \emph{$\bfY$ is representable in $\cC$}.
In case $\bfY=\Spec(L)$, for some $k$-algebra $L$ endowed with an action of $\Gamma$,
we will write $F_L$ for $F_\bfY$ and say that \emph{$L$ is representable in $\cC$} if $\bfY$ is.

We observe that if $\bfY$ is represented in $\mathcal{C}$ by a variety $\bfX$, then the identity map of $\bfX$ corresponds to a geometrization $\bfY \to \bfX$ which is an initial object in $\Geom(\bfY)$, and conversely, if a geometrization $\bfY \to \bfX$ is an initial object in $\Geom(\bfY)$ than $\bfY$ is represented in $\mathcal{C}$ by a variety $\bfX$.

\begin{lemma} \label{lem:disjoint}
Assume $\bfY$ is geometrically reduced. If $\bfY$ is a disjoint union of $\Gamma$-invariant clopen subschemes $\bfY_1\cup \bfY_2$
then $\bfY$ is representable in $\cC$ iff both $\bfY_1$ and $\bfY_2$ are representable in $\cC$.
In this case, $\bfY$ is represented by the dijoint union $\bfX_1\cup\bfX_2$ where, for $i=1,2$ ,$\bfX_i$ represents $\bfY_i$.
\end{lemma}

\begin{proof}
It is clear that if, for $i=1,2$, $\bfY_i$ is represented by $\bfX_i$ then $\bfY$ is represented by $\bfX_1\cup\bfX_2$.
Assume that $\bfY$ is represented by $\bfX$,
thus $\Geom(\bfY)$ has an initial object $\phi:\bfY\to \bfX$.
We let $\phi_i=\phi|_{\bfY_i}$, consider it as a geometrizations of $\bfY_i$ and claim that $E(\phi_i)$ 
is an initial objects in $\Geom(\bfY_i)$,
where $E$ is the functor defined in Proposition~\ref{prop:E}.
Indeed, assuming without lost of the generality that $i=1$, for every geometrization $\phi'_1:\bfY_1\to \bfX'_1$, we consider the corresponding geometrization 
$\phi':\bfY \to \bfX'_1 \cup \Spec(k)$ given by $\phi'|_{\bfY_1}=\phi'_1$ and $\phi'|_{\bfY_2}$ being the trivial geometrization,
and applying the functor $E$ to the morphism of geometrizations of $\bfY_1$ given by 
$\bfX\to \bfX'_1 \cup \Spec(k)$, we obtain a morphism of geometrizations 
realizing $E(\phi_1)$ as an initial object.
\end{proof}

When the base field $k$ varies, we adopt the notation $\Geom_k(\bfY)$ to denote the category of equivariant $k$-geometrization of the $\Gamma$ $k$-scheme $\bfY$ by $k$-varieties in $\cC_k$, emphasizing the dependence on the base field.
For a field extension $k<k'$ and a equivariant $k$-geometrization $\phi:\bfY \to \bfX$,
the extension of scalars $\phi_{k'}:\bfY_{k'} \to \bfX_{k'}$ is an equivariant $k'$-geometrization of $\bfY_{k'}$.
We thus get a faithful functor from $\Geom_k(\bfY)$ to $\Geom_{k'}(\bfY_{k'})$.

\begin{prop}\label{prop:des}
	Let $\bfY$ be a $k$-scheme and $k<k'$ be a finite Galois field extension. If $\bfY_{k'}$ is $k'$-representable then $\bfY$ is $k$-representable. 
\end{prop}
\begin{proof}
Assume $\bfY_{k'}$ is represented in $\cC_{k'}$ by $\bfX$.
	For any $\alpha\in \text{Gal}(k'/k)$ 
 let $\bfY_{k'}^\alpha$ be the scheme $\bfY_{k'}$ with the $k'$-structure twisted by $\alpha$
	and observe that $\bfY_{k'}^\alpha$ is represented in $\cC_{k'}$ by $\bfX^\alpha$, the corresponding twist of $\bfX$. The natural isomorphism $\bfY_{k'} \to \bfY_{k'}^\alpha$ gives an isomorphism of functors $F_{\bfY} \to F_{\bfY_{k'}}$. By Yoneda's Lemma we obtain an isomorphism $\bfX\to \bfX^\alpha$. By Galois descent this yields a $k$-variety $\bar\bfX$ with a $k'$-isomorphism  $\bar\bfX_{k'}\cong \bfX$ and it is straight forward that $\bar\bfX$ represents $\bfY$ in $\cC_k$.
\end{proof}

We will say that a variety $\bfX$ in $\mathcal{C}$ is transitive if  the action of $\bfG_{\bar k}$  on $\bfX_{\bar k}$ is transitive.
We will say that a geometrization $\bfY \to \bfX$ in $\Geom(\bfY)$ is transitive if 
 $\bfX$ is transitive.
 We will denote the subcategory of transitive geometrizations in $\Geom(\bfY)$ by $\Geom^T(\bfY)$
 and note that $\Geom^T(\bfY)$ is a full subcategory of $\Geom^E(\bfY)$.

 Recall that a category is \emph{posetal} if the hom set of every two objects contains at most one element, and a posetal category satisfies the \emph{descending chain condition} if every backwards sequence of morphisms eventually consists of isomorphisms.

 \begin{prop} \label{prop:posetal}
     $\Geom^T(\bfY)$ is a posetal category satisfying the descending chain condition.
 \end{prop}

\begin{proof}
Consider the extension of scalars of $\bfG$ from $k$ to $\bar{k}$, $\bfG_{\bar{k}}$,
and let $\Sub(\bfG_{\bar{k}})$ be the set of reduced $\bar{k}$-algebraic subgroups of $\bfG_{\bar{k}}$.
This is a poset, with respect to inclusion, which satisfies the descending chain condition, by Neotherianity.
Considering it as a posetal category, $\Sub(\bfG_{\bar{k}})$ is equivalent to the category $\Tran(\bfG_{\bar{k}})$, consisting of pointed transitive $\bfG_{\bar{k}}$-varieties defined over $\bar{k}$, thus $\Tran(\bfG_{\bar{k}})$ is a 
posetal category satisfying the descending chain condition. 
We will prove the proposition by constructing a faithful functor from $\Geom^T(\bfY)$ to $\Tran(\bfG_{\bar{k}})$.

We recall that the extension of scalars from $k$ to $\bar{k}$ provides a faithful functor from the category $\Geom^T_k(\bfY)$ to the category $\Geom^T_{\bar{k}}(\bfY_{\bar{k}})$, consisting of transitive equivariant $\bar{k}$-geometrization of $\bfY_{\bar{k}}$.
We also recall that $\bar{k}$ was chosen so that $\bfY$ admits a $\bar{k}$-point
and fix such a point, $\Spec(\bar{k})\to \bfY$.
Given $\phi:\bfY_{\bar{k}}\to \bfX$ in $\Geom^T_{\bar{k}}(\bfY_{\bar{k}})$ we get a $\bar{k}$-point in $\bfX$ by precomposing $\phi$ with the map $\psi$ given by the composition
\[ \Spec(\bar{k}) \to \Spec(\bar{k}\otimes_k \bar{k}) \to \bfY_{\bar{k}},\]
where the first map is the dual to the multiplication map of $\bar{k}$ and the second map is the extension of scalars of our fixed $\bar{k}$-point.
Keeping the geometric point $\phi\circ \psi$ but forgetting $\phi$, gives a forgetful functor from $\Geom^T_{\bar{k}}(\bfY_{\bar{k}})$ to $\Tran(\bfG_{\bar{k}})$, which is easily seen to be fully faithful. 
Thus the composition of functor
\[ \Geom^T(\bfY)\to \Geom^T_{\bar{k}}(\bfY_{\bar{k}}) \to \Tran(\bfG_{\bar{k}}) \]
gives the required faithful functor.
\end{proof}

\begin{thm} \label{thm:erginit}
    If $\bfY$ is zero dimensional, geometrically reduced and we have $\cO(\bfY)^\Gamma=k$ then $F_\bfY$ is represented by a transitive variety in $\cC$.
\end{thm}

In the proof of Theorem~\ref{thm:erginit} we will use the following.

\begin{thm}[Rosenlicht, {\cite[Theorem 2]{ros56}}]\label{thm:ros}
Fix a variety $\bfX$ in $\mathcal{C}$.
Then there exist an open dense $\bfG$-invariant subset $\bfU\subset \bfX$, algebraic variety $\bfX'$ and a dominant $\bfG$-equivariant map $\phi:\bfU\to \bfX'$ (all defined over $k$) such that the reductions of its fibers are transitive. 
\end{thm}

\begin{prop} \label{prop:T}
    If $\bfY$ is zero dimensional, geometrically reduced and we have $\cO(\bfY)^\Gamma=k$ then the inclusion of categories $\Geom^T(\bfY) \subseteq \Geom(\bfY)$
has a right adjoint functor.
\end{prop}

\begin{proof}
    In view of Proposition~\ref{prop:E}, it is enough to show that the inclusion of categories $\Geom^T(\bfY) \subseteq \Geom^E(\bfY)$
has a right adjoint functor.
For this, it is enough to show that for every $\bfX$ in $\Geom^E(\bfY)$ there exists a unique open $\bfG$-invariant dense subset $\bfU\subseteq \bfX$ such that $\bfU_{\bar k}$ is transitive and the map
$\bfY\to \bfX$ factors via the inclusion $\bfU\hookrightarrow \bfX$. 
The uniqueness is obvious by transitivity, since two open dense subsets have to intersect non-trivially.
We will show the existence. 

Let $\bfU\subseteq \bfX$ and $\phi:\bfU\to\bfX'$  be given by Rosenlicht Theorem, \ref{thm:ros}.
We let $e_{\bfU}$ be the idempotent corresponding to $\bfU$, given by Proposition~\ref{prop:clopen}.
Since it is $\Gamma$-invariant and 
$\cO(\bfY)^\Gamma=k$, we get that $e_{\bfU}$ equals either 0 or 1, and since the geometrization is essential, we conclude that $e_{\bfU}=1$.
Therefore, the geometrization $\bfY \to \bfX$ indeed factors via the inclusion $\bfU\hookrightarrow \bfX$. 
In particular, we get a map $\bfY \to \bfX'$.
By Proposition~\ref{pro:zdaffine} this map factors via the affine scheme $\Spec(\cO(\bfY))$.
Since the action of $\bfG$ on $\bfX'$ is trivial, the map $\Spec(\cO(\bfY))\to \bfX'$ factors via a map $\spec k=\spec (\cO(\bfY)^\Gamma) \to \bfX'$. The dominance of this map implies that  $\spec k \cong \bfX'$.
We conclude that $\bfU$ is a transitive variety.
This finishes the proof.
\end{proof}

\begin{proof}[Proof of Theorem~\ref{thm:erginit}]
We need to show that the category $\Geom(\bfY)$ has an initial object which is transitive.
In view of Proposition~\ref{prop:T}, it is enough to show that the category $\Geom^T(\bfY)$ has an initial object.
We note that this category has product, by Remark~\ref{rem:prod}, and it is a posetal category satisfying the descending chain condition, by Proposition~\ref{prop:posetal}.
By the descending chain condition it has a minimal object, that is an object such that every morphism pointing at it is an isomorphism, and by having products, such a minimal object must be initial.
\end{proof}

\section{Proof of the main results} \label{sec:proofs}

\subsection{Proof of Theorem~\ref{thm:E}} \label{ssec:thmE}
 We set $L=\cO(\bfY)$.
 We observe that $L$ is geometrically reduced. 
Indeed, this is a subalgebra of the product of local algebras $\prod_{y\in \bfY} L_y$, each of which is geometrically reduced by \cite[Lemma 035W]{SP}.
 By Proposition~\ref{pro:zdaffine}, $L$ is a zero dimensional algebra
and every geometrization of $\bfY$ factors uniquely via the canonical map $\bfY\to \Spec(L)$.
The equivalence of the first two bullets of Theorem~\ref{thm:E}, as well as the ``moreover" part follow.
The equivalence of the last two bullets follows from Lemma~\ref{lem:finitealg}.
Thus we are reduced to prove that $L$ is representable in $\cC$ iff $L^\Gamma$ is a finite sum of finite  separable field extensions of $k$.

\begin{proof}[Proof of the ``only if" part]
We assume $L$ is a geometrically reduced zero dimensional algebra which is representable in $\cC$ and argue to show that $L^\Gamma$ is a finite sum of finite field extensions of $k$.
We let $\Spec(L)\to \bfX$ be the initial object in the category of geometrizations of $\Spec(L)$.
We let $k<\bar{k}$ be an algebraically closed field extension.
We let $n$ be the number of absolute connected components of $\bfX$, that is the number of connected components of $\bfX_{\bar k}$.

We claim that $L^\Gamma$ has finitely many idempotents.
Assume by contradiction that $L^\Gamma$ has infinitely many idempotents.
Then we can find non-trivial orthogonal idempotents $e_0,\ldots,e_n\in L^\Gamma$, thus by Lemma~\ref{lem:disjoint},
$X=\bfX_0\cup \cdots \cup\bfX_n$ is a direct union of clopen subvarieties. 
This is a contradiction, thus indeed, $L$ has finitely many idempotent. 

Using Lemma~\ref{lem:disjoint} again, we are reduced to the case that $L^\Gamma$ has a single idempotent, which we now assume.
By Lemma~\ref{lem:LGamma}, $L^\Gamma$ is a geometrically reduced zero dimensional $k$-algebra,
and it follows that $L^\Gamma$ is a separable field extension of $k$.

We claim that the $L^\Gamma$ is an algebraic extension of $k$.
Assuming this is not the case, we fix $t\in L^\Gamma$ which is not algebraic over $k$ and consider $A=k[t]<k(t)<L^\Gamma$.
By the definition of $\bf X$, the geometrization $\Spec(L)\to \Spec(A)$ factors via a map $\bfX\to \Spec(A)$. 
Let $z$ be a $\bar k$-point in the image of the latter map. Let $p\in k[x]$ be the minimal polynomial of $z$ and denote $B=A[p^{-1}]<L$.
The corresponding geometrization $\Spec(L)\to \Spec(B)$ factors via a map $\bfX\to \Spec(B)$
and we conclude that the map $\bfX\to \Spec(A)$ factors via $\Spec(B)\to \Spec(A)$.
Since $z$ is not in the image of the latter, we get a contradiction.
We conclude that, indeed, $L^\Gamma$ is an algebraic extension of $k$.

To conclude the proof we need to show that $L^\Gamma$ is a finite extension of $k$.
In fact, we will show that $[L^\Gamma:k]\leq n$.
Otherwise, we may find a finite extension $k'<L^\Gamma$ with $[k':k]=m>n$
and consider the dominant geometrization $\Spec(L)\to \Spec(k')$
which factors via a dominant map $\bfX \to \Spec(k')$,
but $\spec(k')$ consists of $m$ $\bar k$-points, as $\bar k\otimes_k k'\cong \oplus_{i=1}^m \bar k$.
This contradiction finishes the proof.
\end{proof}

\begin{proof}[Proof of the ``if" part]
  We have that $L^\Gamma$ is a finite sum of finite separable field extensions of $k$,
$L^\Gamma \cong \oplus_{i=1}^n k_i$, with corresponding idempotents $e_i$.
We need to show that $L$ is representable in $\cC$.
We will consider the following case redaction.

\medskip
\noindent {\bf Case 1:} $L^\Gamma=k$.\\
    This case follows from Theorem~\ref{thm:erginit}.
    
\medskip
\noindent {\bf Case 2:} For every $i=1,\ldots,n$, $k_i\cong k$.\\
    Note that $e_i$ are idempotents in $L$ and $1=\sum e_i$.
    Denoting $L_i=e_iL$, we obtain a $\Gamma$-invariant decomposition $L=\bigoplus L_i$.
    Using Lemma~\ref{lem:disjoint} we are reduced to the previous case. 
    
\medskip
    \noindent {\bf Case 3:} $n=1$, that is $L^\Gamma$ is a field.\\
    Let $k'$ a finite Galois extension of $k$ containing $L^\Gamma$ and set $L'=k' \otimes_{k} L$. 
    By Proposition \ref{prop:des} it is enough to show that  $L'$ is $k'$-representable. 
    It is obvious that $L'$ is a geometrically reduced $k'$-algebra and it is zero dimensional by Lemma~\ref{lem:zdext}. By the $k'$-algebras isomorphisms $(L')^\Gamma\cong k' \otimes_{k} L^\Gamma\cong \oplus_{i=1}^{\dim_{k}L^\Gamma} k'$, we get a reduction to the previous case.
    
\medskip
    \noindent {\bf Case 4:} The general case.\\
    Follows from the previous case in the same way that Case 2 follows from Case 1.
\end{proof}

\subsection{Proof of Proposition~\ref{prop:field}}\label{ssec:ded}

We first note that $K$ is zero dimensional and geometrically reduced.
Indeed, it is a direct sum of the fields  of $k$-rational functions over the irreducible components of $\bfX_0$, and each of this is zero dimensional, as is a field, and it is geometrically reduced by applying \cite[Lemma 035W]{SP} to the generic point. 

Next, we claim that $K^G=k$.
We choose an algebraically closed field extension $k<\bar{k}$, consider the extension of scalars $(\bfX_0)_{\bar{k}}$ and the associated algebra of rational functions $K_1=\bar{k}((\bfX_0)_{\bar{k}})$.
Identifying $\bar{k} \otimes_k K$ as a subalgebra of $K_1$, we get the sequence of inclusions
\[ \bar{k} \subseteq \bar{k} \otimes_k K^G \subseteq (\bar{k} \otimes_k K)^G \subseteq K_1^G. \]
The domain of definition of every function in $K_1^G$ is an open $G$-invariant subset of $\bfX_0$,
thus it is the entire space, by the density of $G$ in $\bfG$ and transitivity.
We conclude that 
\[ K_1^G = \cO_{\bar{k}}(\bfX_0)^G = \cO_{\bar{k}}(\bfX_0)^{\bfG(\bar{k})}= \bar{k}, \] 
where the last two equations follow again by the density of $G$ in $\bfG$ and transitivity.
It follows that $\bar{k} = \bar{k} \otimes_k K^G$ and we conclude that indeed, $K^G=k$.

   By Theorem~\ref{thm:E} the functor $F_K$ is representable by a $\bf G$-variety $\bf X_1$, 
   thus we have an isomorphism of functors  $$\psi: \Mor_{\cC}(\bfX_1,\cdot)\to \Mor(\spec(K),\cdot).$$ 
   We denote $\phi:=\psi(\text{Id}_{\bfX_1})$ and get that $\psi$ is given by precomposition with $\phi$.
   
   So we have a $\bf G$-equivariant map $\nu:\bf X_1 \to \bf X_0$ such that the map $\spec(K) \to \bf X_0$ factors via $\bf X_1$. 
   Thus there exists an open $\bf U \subseteq \bf X_1$  s.t.
   \begin{itemize}
       \item $\phi$ gives an isomorphism of $k(U)$ with $K$.
       \item $\nu|_{\bf U}:\bf U\to \nu(\bf U)$ is an isomorphism.
   \end{itemize}
    Since $\bf X_0$ is 
   $\bf G$-transitive, we can assume that 
   $\nu(\bf U)=\bf X_1$. So $\nu$ admits a section $s$. 
   
   We have $s \circ \nu \circ \phi=\phi$. So the fact that $\psi$ is  a bijection proves that $s \circ \nu=\text{Id}_{\bfX_1}$. This implies that $\nu$  is an isomorphism, and we are done.

\section{Algebraic representations of ergodic actions} \label{sec:meas}

In this section we relate Theorem~\ref{thm:erginit} to some results explored in \cite{BF20} and \cite{BaderFurman}.
We first review some basic notions from measure theory.

	\begin{enumerate}
	\item By a Borel space we mean a set equipped with a choice of a sub-$\sigma$-algebra of its power set, whose elements we call Borel or measurable subsets.
		\item By a morphism, or a Borel map, of  Borel spaces we mean a map under which the preimage of Borel set is a Borel set.
		\item
		This gives a structure of a category on the collection of Borel spaces. We denote this category by $\mathcal B$.
		\item A Borel embedding is a Borel map which image is a Borel subset and it is an isomorphism onto its image.
		\item The category $\mathcal B$ has finite products, given by the product $\sigma$-algebra on the 
		set theoretical product, and fibered products which Borel embed into the corresponding products.
		\item \label{item:top} Endowing a topological space with the $\sigma$-algebra generated by its collection of open sets gives a functor $\Top\to \cB$. 
		This functor preserves finite products and finite fibered products of second countable topological spaces.
		\item 
		A measure on a Borel space is a $[0,\infty)$-valued $\sigma$-additive function defined on its $\sigma$-algebra of Borel subsets.
		The kernel of a measure is a $\sigma$-ideal. 
    	Two measures on a Borel space are said to be equivalent if they have the same kernels.
		\item  By a measured space we will mean a Borel space endowed with a choice of a $\sigma$-ideal
		of Borel subsets, elements of which will be called null-sets.
		The complement of a null set is said to be conull.
		An example of a measured space is obtained by choosing an equivalence class of measures on a Borel space.
		\item
		Given a measured space $\Omega$ and a Borel space $U$, 
		a measured function from $\Omega$ to $U$ is a choice of a conull subset $\Omega_0\subseteq \Omega$
		and a Borel map $\Omega_0\to U$.
		Two measured functions are said to be equivalent if they agree on conull subsets of their domains of definition.
		\item An equivalence class of measured functions is called a measured map.
		The set of all measured maps from the measured space $\Omega$ to the Borel space $U$
		is denoted $\Maps(\Omega,U)$.
		\item Given two measured spaces $\Omega$ and $\Omega'$,
		a measured function from $\Omega$ to $\Omega'$ is said to be a measured morphism if
		the preimage under any representing measured function of every null set of $\Omega'$ is a null set of $\Omega$.
		The set of all such measured morphism is denoted $\Mor(\Omega,\Omega')$.
		\item 
		This gives a structure of a category on the collection of measured spaces. We denote this category by $\mathcal M$.
		\item Even though every measured space has an underlying structure of a Borel space, forgetting the $\sigma$-ideal of null sets does not give a functor from $\mathcal B$ to $\mathcal M$, as it does not preserve morphisms. 
		However, $\Maps(\cdot,\cdot)$ gives a correspondence, or a pro-functor, from $\mathcal B$ to $\mathcal M$.
	\end{enumerate}

By a Borel ring we mean a ring object in the category of Borel spaces, that is a Borel space $k$ 
endowed with two Borel maps $k\times k \to k$ which satisfy the usual axioms of addition and multiplication. 
By item~(\ref{item:top}) above, examples of a Borel rings are obtained by considering topological rings which are second countable.

A Borel ring $k$ which underlying ring is a field
is said to be a Borel field if the first coordinate projection
\[ \{(x,y) \in k^2 \mid xy=1\} \to k, \quad (x,y) \mapsto x \]
is a Borel embedding, i.e an isomorphism onto $k^\times$.
Equivalently, this means that the inverse map is Borel on $k^\times$.

\begin{example}
Given a field $k$ and an absolute value on $k$ which gives rise to a complete and separable metric,
the corresponding Borel structure on $k$ makes it a Borel field.
\end{example}

\begin{prop} \label{prop:borelfunct}
    Given a Borel field $k$ there exists a unique functor from $k$-varieties to Borel spaces, $B:\cV \to \cB$, such that
    \begin{itemize}
        \item $B$ preserves products and fiber products.
        \item $B$ takes open immersions to Borel embeddings.
        \item The image under $B$ of $\mathbb{A}^1$ is the Borel space underlying $k$.
        \item The composition of $B$ with the forgetful functor $\cB \to \Sets$ is the functor of points.
    \end{itemize}
\end{prop}

\begin{proof}
    The proof is fairly standard (once the precise statement is made). For concreteness, one can follow the proofs of \cite[Propositins~2.1 and 3.1]{Conrad12}
    which deal with topological rings, making the obvious adjustments.
\end{proof}

In view of the last bullet above, given a $k$-variety $\bfX$, $B(\bfX)$ is the set $\bfX(k)$ endowed with a certain Borel structure.
Therefore, we will simply use the term $\bfX(k)$ instead of $B(\bfX)$ when it is clear from the context that we regard it as a Borel space.

\begin{prop} \label{prop:finiteextborel}
    Given a Borel field $k$ and a finite field extension $k<k'$, there is a Borel structure on $k'$ making it a Borel field and $k\hookrightarrow k'$ a Borel embedding.
\end{prop}

The Borel structure alluded to on $k'$ is nothing but the product Borel structure associated with an identification $k'\cong k^n$ as $k$-vector spaces. 
However, to avoid a choice we present this in the proof below in a more abstract way.

\begin{proof}
We let $\bfX$ be the Weil restriction of scalars of $\mathbb{A}^1$ from $k'$ to $k$, thus $\bfX(k)\cong k'$. 
Using Proposition~\ref{prop:borelfunct}, we get a Borel structure on $k'$ with respect to which the addition and multiplication maps are Borel. This makes $k'$ into a Borel ring.
Since $\mathbb{G}_m\to \mathbb{A}^1$ is an open immersion, open immersions are preserved by restriction of scalars (see \cite[Proposition 7.6/2]{NeronModels}) and $B$ takes open immersions to Borel embeddings, we get that $k'$ is a Borel field.
\end{proof}

Given a Borel field $k$, we get a functor from measured spaces to $k$-algebras, $\Omega \mapsto \Maps(\Omega,k)$.
We denote this functor by $L$, that is we set $L(\Omega)=\Maps(\Omega,k)$.
Composing with the spectrum functor we get a functor from measured spaces to schemes, $\Omega \mapsto \Spec(L(\Omega))$.
In fact, this functor takes values in geometrically reduced zero dimensional schemes.

\begin{lemma} \label{lem:Lgrzd}
For every $\Omega$ in $\cM$, $L(\Omega)$ is a geometrically reduced zero dimensional algebra and $\Spec(L(\Omega))$ is a geometrically reduced zero dimensional scheme.
\end{lemma}

\begin{proof}
The fact that $L(\Omega)$ is a zero dimensional algebra follows at once from Lemma~\ref{lem:zdidem}. Indeed, for a map in $L(\Omega)$,
the characteristic function of its support (which is well defined up to null sets) is a supporting idempotent.
The fact that it is reduced is immediate.
To see that it is geometrically reduced, it is enough by \cite[Lemma 030V]{SP} to show that $k'\otimes_k L$ is reduced for every finite inseparable extension $k<k'$.
But a finite extension $k'$ is a Borel field by Proposition~\ref{prop:finiteextborel}, and one sees easily that
$k' \otimes_k \Maps(\Omega,k) \cong \Maps(\Omega,k')$, which is indeed reduced.
\end{proof}

The main result of this section is the following.

\begin{theorem}\label{thm:geom.rep}
	Let $k$ be a Borel field.
	The functors $\Spec\circ L:\cM \to \Schemes$ and $B:\cV \to \cB$ considered above
	are adjoint in the following sense. 
	For every $\bfX$ in $\cV$ and $\Omega$ in $\cM$, there is a bijection
	\[ \Maps(\Omega,\bfX(k))\cong \Mor(\Spec(L(\Omega)),\bfX) \] 
	and this bijection depends functorially on both $\bfX$ and $\Omega$.	 
\end{theorem}


\begin{proof}
We will construct a map $\phi:\Maps(\Omega,\bfX(k))\to \Mor(\Spec(L(\Omega)),\bfX)$ and show that it has the desired properties.
We will do it in several steps.

We first consider the case in which $\bfX$ is affine.
		In this case we identify $\Mor(\cO(\bfX) ,L(\Omega)) \cong  \Mor(\spec(L(\Omega)),\bfX)$,
		and define $\phi:\Maps(\Omega,\bfX(k))\to \Mor(\cO(\bfX) ,L(\Omega))$
		to be the obvious pull back.
		The functoriality of $\phi$ is obvious, and we are left to construct an inverse map. 
		We chose a closed embedding $\bfX\subseteq \bA^n$. Each coordinate map $\bA^n\to \bA^1$ gives a map 
		\[ \Mor(\cO(\bfX) ,L(\Omega)) \to \Mor(\cO(\bA^1) ,L(\Omega)) \cong L(\Omega) = \Maps(\Omega,k).\]
		Combining the coordinate maps, we get  
		\[ \Mor(\cO(\bfX) ,L(\Omega)) \to \Maps(\Omega,k)^n \cong \Maps(\Omega,k^n),\]
		which image is in $\Mor(\Omega,\bfX(k))\subseteq \Mor(\Omega,k^n)$.
		This gives a map 
		\[\Mor(\cO(\bfX) ,L(\Omega)) \to \Maps(\Omega,\bfX(k)) \]
		which is easily seen to be an inverse to $\phi$.
		This completes the proof of the theorem in case $\bfX$ is affine.
		
		We now turn to the construction of $\phi$ in the general case.
		Fix $\alpha \in \Maps(\Omega,\bfX(k))$ and a Borel representative $\alpha_0:\Omega_0 \to \bfX(k)$.
		Chose an affine stratification $\bfX=\bigcup_{i=1}^n \bfX_i$ and set $\Omega_i= \alpha_0^{-1}(\bfX_i(k))$ and $\alpha_i=\alpha_0|_{\Omega_i}:\Omega_i \to \bfX_i(k)$.  
		By the already proven affine case, we get corresponding morphisms $\beta_i=\phi(\alpha_i): \spec(L(\Omega_i)) \to \bfX_i$. Note that $\spec(L(\Omega_i))$ is naturally identified with  a clopen subset of $\spec(L(\Omega))$, and moreover $\spec(L(\Omega_i))$ form a disconnected cover of $\spec(L(\Omega))$. Thus the collection $\{\beta_i\}$ gives a map $\beta:\spec(L(\Omega)) \to \bfX$.
		
		We would like to define $\phi(\alpha)=\beta.$ For this we need to show that $\beta$ does not depend on the choices of $\alpha_0$ and the stratification $\bfX=\bigcup_{i=1}^n \bfX_i$. 
		The independence on the choice of $\alpha_0$ is clear.
		Given two stratifications as above, since they have a joint refinement, we may assume that one of them refines the other. Thus we may reduce to the case when $\bfX$ is affine and one of the stratifications is trivial. This indeed follows from the functoriality of $\phi$ in the affine setting and we are thus done constructing $\phi$.
		
		The functoriality of $\phi$ with respect to $\Omega$ is obvious. 
		We now argue to show its functoriality with respect to $\bfX$.
		Fixing $\alpha \in \Maps(\Omega,\bfX(k))$ and $\nu:\bfX\to  \bfY$ a morphism in $\cC$, we need to show that the following diagram is commutative:
		\begin{equation}\label{eq:fun}
			\xymatrix{
				\spec(L(\Omega))\ar[rr]^{\quad \quad \,\,\phi(\alpha)}  \ar[drr]_{\phi(\nu(k)\circ \alpha)} &&\bfX \ar[d]^{\nu}\\
				& & \bfY}.	  
		\end{equation}
		Choose an affine stratification $\bfY=\bigcup_{i=1}^n \bfY_i$ and, for each $i$, chose an  affine stratification $\nu^{-1}(\bfY_i)=\bigcup_{j=1}^{n_i} \bfX_{ij}$.
		We get an affine stratification $\bfX=\bigcup_{i,j} \bfX_{ij}$.
		Using these stratifications for the construction of $\phi(\nu(k)\circ \alpha)$ above and $\phi(\alpha)$ correspondingly, the commutativity of diagram \eqref{eq:fun} follows the functoriality in the affine case.

	We are left to show that $\phi$ is bijective.
		We will chose an affine stratification $\bfX=\bigcup_{i=1}^n \bfX_i$ and construct a map $\psi:\Mor(\spec(L(\Omega)),\bfX) \to \Maps(\Omega,\bfX(k))$. 
		This construction will depend a priori on the chosen stratification,
		however, we will show that it is an inverse to $\phi$, thus a posteriori it will not.
		Fixing $\beta \in \Mor(\spec(L(\Omega)),\bfX)$, we will construct $\psi(\beta)$.
		We set $\bf Z_i=\beta^{-1}(X_i)$. By Proposition~\ref{prop:clopen}, $\bf Z_i$ is clopen in $\spec(L(\Omega))$, thus $\spec(L(\Omega))=\bigcup \bf Z_i$ is a clopen disjoint cover. This gives a direct sum decomposition $L(\Omega)=\bigoplus \cO(\bfZ_i)$, and a corresponding decomposition of the unity in $L(\Omega)$ as a sum of idempotents, $1=\sum e_i$, see Remark~\ref{rem:projmeas}. Each $e_i$ is the characteristic function of some set $\Omega_i\subseteq \Omega$ (which is well defined only up to a measure $0$ set) and we have $\Spec(L(\Omega_i))=\bfZ_i$.
		We set $\beta_i:=\beta|_{\bfZ_i}$ and consider it as a morphism from $\Spec(L(\Omega_i))$ to $\bfX_i$. 
		Using the already proven bijectivity of $\phi$ in the affine case, we set $\alpha_i=\phi^{-1}(\beta_i)$ and define $\alpha: \Omega\to \bfX(k)$ to be the unique (up to equivalence) map satisfying $\alpha|_{\Omega_i}=\alpha_i$. Finally we set $\psi(\beta)=\alpha$. 
		Using the same stratification $\bfX=\bigcup_{i=1}^n \bfX_i$ for the construction of $\phi(\alpha)$, it is easy to see that $\phi(\alpha)=\beta$.
		We have thus constructed a map $\psi$ which is inverse to $\phi$.
		We conclude that $\phi$ is indeed bijective. This completes the proof.
\end{proof}

We now fix a group $\Gamma$, a measured space $\Omega$ and a group homomorphism $\theta:\Gamma \to\Aut_\cM(\Omega)$.
We also fix a Borel filed $k$, a $k$-algebraic group $\bfG$
and a group homomorphism $\rho:\Gamma\to \bfG(k)$.
We consider the category $\mathcal{C}$ of 
$k$-varieties (see Definition~\ref{def:variety}) 
on which $\bfG$ acts $k$-morphically
and their $\bfG$-equivariant $k$-morphisms.

\begin{defn} [{cf. \cite[Definition~4.1]{BF20}}]
An algebraic representation of $\Omega$ is a pair $(\bfX,f)$ where $\bfX$ is a variety in $\cC$
and $f$ is a measured map in $\Maps(\Omega,\bfX(k))$ which is $\Gamma$ equivariant, that is for every $\gamma\in \Gamma$, $f\circ \theta(\gamma)=\rho(\gamma)\circ f$ in $\Maps(\Omega,\bfX(k))$.

A morphism between algebraic representations $(\bfX,f)$ and $(\bfX',f')$ is a morphism in $\cC$, $\alpha:\bfX\to \bfX'$, such that $f'=\alpha(k)\circ f$ in $\Maps(\Omega,\bfX'(k))$,
that is the diagram 
		$$\xymatrix{
			\Omega\ar[r]^{f}  \ar[dr]_{f'} &\bfX(k) \ar[d]^{\alpha(k)}\\
			&\bfX'(k)}.$$
		is commutative. 	
\end{defn} 

We observe that the collection of representation of $\Omega$ forms a category.
Theorem~\ref{thm:geom.rep} gives the following.

 \begin{cor}
 	the category of algebraic representations of $\Omega$ is equivalent to the category of equivariant geometrizations of $\spec(L(\Omega))$.
 \end{cor}

We say that $\Omega$ is $\Gamma$-ergodic if $L(\Omega)^\Gamma$ consists only of the constant maps
$\Omega\to k$, that is $L(\Omega)^\Gamma\cong k$.
By Lemma~\ref{lem:Lgrzd} and Theorem~\ref{thm:erginit} we get the following.

\begin{cor}[{cf. \cite[Theorem~4.3]{BF20}}]
If $\Omega$ is $\Gamma$-ergodic then the category of algebraic representations of $\Omega$ has an initial object which is transitive.
\end{cor}

\bibliographystyle{alpha}
\bibliography{Ramibib}

\begin{thebibliography}{{Sta}18}

\bibitem[BF20]{BF20}
Uri Bader and Alex Furman.
\newblock Super-rigidity and non-linearity for lattices in products.
\newblock {\em Compos. Math.}, 156(1):158--178, 2020.

\bibitem[BF22]{BaderFurman}
Uri Bader and Alex Furman.
\newblock An extension of {M}argulis's {S}uperrigidity theorem.
\newblock In {\em Dynamics, geometry, number theory---the impact of {M}argulis
  on modern mathematics}, pages 47--65. Univ. Chicago Press, Chicago, IL,
  [2022] \copyright 2022.

\bibitem[BLR90]{NeronModels}
Siegfried Bosch, Werner L\"{u}tkebohmert, and Michel Raynaud.
\newblock {\em N\'{e}ron models}, volume~21 of {\em Ergebnisse der Mathematik
  und ihrer Grenzgebiete (3) [Results in Mathematics and Related Areas (3)]}.
\newblock Springer-Verlag, Berlin, 1990.

\bibitem[BR06]{BrRi}
Jim Brewer and Fred Richman.
\newblock Subrings of zero-dimensional rings.
\newblock In {\em Multiplicative ideal theory in commutative algebra}, pages
  73--88. Springer, New York, 2006.

\bibitem[Con12]{Conrad12}
Brian Conrad.
\newblock Weil and {G}rothendieck approaches to adelic points.
\newblock {\em Enseign. Math. (2)}, 58(1-2):61--97, 2012.

\bibitem[Ros56]{ros56}
Maxwell Rosenlicht.
\newblock Some basic theorems on algebraic groups.
\newblock {\em Amer. J. Math.}, 78:401--443, 1956.

\bibitem[{Sta}18]{SP}
The {Stacks Project Authors}.
\newblock \textit{Stacks Project}.
\newblock \url{https://stacks.math.columbia.edu}, 2018.

\end{thebibliography}
\end{document}